\documentclass[runningheads, 11pt]{llncs}

\usepackage{amssymb}
\usepackage{graphicx}
\usepackage{mathrsfs}
\usepackage{hyperref}
\usepackage{color}
\usepackage{url}
\usepackage{cite}

\usepackage{times}
\usepackage{microtype}
\usepackage{amssymb}
\usepackage{mathtools}
\usepackage{tikz}
\usepackage{wasysym}
\usepackage{centernot}
\usepackage{color}
\usepackage{setspace}
\usepackage{appendix}
\usepackage{booktabs}
\usepackage{multirow}

\urldef{\mailsa}\path|liliping@amss.ac.cn, wjsun@tongji.edu.cn.|    
\newcommand{\keywords}[1]{\par\addvspace\baselineskip
\noindent\keywordname\enspace\ignorespaces#1}

\def\sE{{\mathscr E}}
\def\sF{{\mathscr F}}

\def\bR{{\mathbb{R}}}
\def\sG{{\mathscr G}}
\def\sA{{\mathscr A}}

\def\bR{{\mathbb R}}

\def\EE{{\mathscr E}}
\def\FF{{\mathscr F}}
\def\sE{{\mathscr E}}
\def\sF{{\mathscr F}}
\def\sA{{\mathscr A}}

\def\bH{\mathbf{H}}

\def\b0{{\textbf{0}}}

\def\<{\langle}
\def\>{\rangle}

\begin{document}

%\mainmatter  % start of an individual contribution

% first the title is needed
\title{On trace of Brownian motion on the boundary of a strip}

% a short form should be given in case it is too long for the running head
\titlerunning{Trace of Brownian motion}

% the name(s) of the author(s) follow(s) next
%
% NB: Chinese authors should write their first names(s) in front of
% their surnames. This ensures that the names appear correctly in
% the running heads and the author index.
%

%\author{Liping Li}
%\institute{RCSDS, HCMS, Academy of Mathematics and Systems Science, Chinese Academy of Sciences, Beijing 100190, China.}

\author{Liping Li${}^{1,2}$ and Wenjie Sun${}^{3}$%
\thanks{The first named author is partially supported by NSFC (No. 11688101, No. 11801546 and No. 11931004), Key Laboratory of Random Complex Structures and Data Science, Academy of Mathematics and Systems Science, Chinese Academy of Sciences (No. 2008DP173182), and Alexander von Humboldt Foundation in Germany.\\
Email: \mailsa}%
}
\institute{${}^1$RCSDS, HCMS, Academy of Mathematics and Systems Science,\\ Chinese Academy of Sciences, Beijing, China\\
${}^2$Bielefeld University, Bielefeld, Germany \\
${}^3$School of Mathematical Sciences, Tongji University,  China.}

\authorrunning{L. Li and W. Sun}
% (feature abused for this document to repeat the title also on left hand pages)

% the affiliations are given next; don't give your e-mail address
% unless you accept that it will be published

%\institute{Bielefeld University, Bielefeld, Germany}
%\mailsc\\
%\url{http://www.springer.com/lncs}

%
% NB: a more complex sample for affiliations and the mapping to the
% corresponding authors can be found in the file "llncs.dem"
% (search for the string "\mainmatter" where a contribution starts).
% "llncs.dem" accompanies the document class "llncs.cls".
%

%\toctitle{Lecture Notes in Computer Science}
%\tocauthor{Authors' Instructions}
\maketitle

\begin{abstract}
The trace of a Markov process is the time changed process of the original process on the support of the Revuz measure used in the time change.  In this paper,  we will concentrate on the reflecting Brownian motions on certain closed strips.  On one hand,  we will formulate the concrete expression of the Dirichlet forms associated with the traces of such reflecting Brownian motions on the boundary.  On the other hand,  the limits of these traces as the distance between the upper and lower boundaries tends to $0$ or $\infty$ will be further obtained.  

\keywords{Trace Dirichlet forms,  Time change,  Reflecting Brownian motion,  Mosco convergence.}
\end{abstract}

\section{Introduction}

The trace of a Markov process is the time changed process of the original process on the support of the Revuz measure used in the time change.  It has been studied by many researchers.  In \cite{CFY06} and \cite{FHY04}, the authors give a complete characterization of the traces of arbitrary symmetric Markov processes,  in terms of the Beurling-Deny decomposition of their associated Dirichlet forms and of the Feller measures of the processes.  

In this paper we will concentrate on the reflecting Brownian motion $B^T$ on a closed strip $T=\bR\times [0,\pi]$ and formulate the concrete expression of its trace $\check{B}^T$ on the boundary $\partial T=\bR\times \{0,\pi\}$ by means of the characterization in \cite{CFY06} and \cite{FHY04}.  The crucial step is to figure out the Feller measure for $\check{B}^T$.  In some previous works like \cite{AS82} and \cite[Example~2.1]{FHY04},   concrete Feller measures are obtained only for traces on certain compact hypersurfaces,  and these Feller measures are usually given by the derivatives of the Poisson kernels along the inward normal direction;  see \eqref{eq:21}.  One of the main results,  Theorem~\ref{THM1},  concludes that this formulation of the Feller measure involving Poisson kernel also holds for this unbounded $\partial T$.  With this Feller measure at hand,  we will further derive the expression of the Dirichlet form associated with $\check{B}^T$ in Theorem~\ref{THM2}.    

After a spatial transformation and scaling on $\check{B}^T$,  we can also formulate the trace $\check{B}^\ell$ of reflecting Brownian motion $B^\ell$ on $\bar{\Omega}_\ell:=\bR\times [-\frac{\pi}{2}\ell, \frac{\pi}{2}\ell]$ on the boundary $\partial \bar{\Omega}_\ell$ for a constant $\ell>0$.  Another purpose of this paper is to study the limit of $\check{B}^\ell$ as $\ell\downarrow 0$ or $\ell\uparrow \infty$.  As we will see in its Dirichlet form characterization,  $\check{B}^\ell$ enjoys jumps taking place between the two components of $\partial \bar{\Omega}_\ell$.  Then Theorem~\ref{THMD5} tells us that when $\ell\downarrow 0$,  this kind of jump becomes so frequent that its energy diverges.  In other words,  no sensible limits as $\ell\downarrow 0$ can be obtained.  However when $\ell\uparrow \infty$,  the connection between the two components of $\partial \bar{\Omega}_\ell$ is cut off,  and $\check{B}^\ell$ converges to a Markov process,  consisting of two distinct Cauchy processes on the upper and lower real lines respectively,  in a certain sense. 

The approach to prove the main results in this paper is by virtue of the theory of Dirichlet forms.  For the terminologies and notations in the theory of Dirichlet forms, we refer to \cite{CF12} and \cite{FOT11}.   The symbol $\lesssim$ (resp. $\gtrsim$) means that the left (resp. right) term is bounded by the right (resp. left) term multiplying a non-essential constant.

\section{Trace Dirichlet form on the boundary of a strip}\label{APB}

Let us consider a closed strip $T:=\mathbb{R}\times[0,\pi]$ in $\mathbb{R}^2$, and define 
\[
\begin{aligned}
  \mathscr{F}&=H^1(T),\\
 \mathscr{E} (u,u) &=
 \frac{1}{2}\int_{T} |\nabla u|^2dx,\quad u\in \mathscr{F},
 \end{aligned}
\]
where $H^1(T):=H^1(\mathring{T})=\left\{ u\in L^2(\mathring{T}): \frac{\partial u}{\partial x_i} \in L^2(\mathring{T}), i=1,2\right\}$ with $\mathring{T}:=\bR\times (0,\pi)$ and $\frac{\partial u}{\partial x_i}$ is the derivative of $u$ on $\mathring{T}$ in the sense of Schwartz distribution.  
Then $(\mathscr{E}, \mathscr{F})$ is a regular and recurrent Dirichlet form on $L^2(T)$ associated with the reflecting Brownian motion $B^T:=(B^T_t)_{t\geq 0}$ on $T$; see, e.g., \cite[\S2.2.4]{CF12}.  Let $\sigma$ be the Lebesgue measure on  $\partial T=\bR\times \{0,\pi\}$.  Note that $\sigma$ can be regarded as a measure on $T$ by imposing $\sigma(\mathring{T})=0$. The following lemma shows that $\sigma$ is a Radon smooth measure with respect to $B^T$.  

\begin{lemma}
$\sigma$ is smooth with respect to $B^T$.
\end{lemma}
\begin{proof}
Take $n>0$ and set $\sigma_n:=1_{\{x=(x_1,x_2)\in \partial T: |x_1|<n\}}\cdot \sigma$.  By means of the trace theorem,  we have for any $u\in \mathscr{F}\cap C_c(T)$,
\[
	\int_{T} |u(x)|\sigma_n(dx)\leq \left(\int_{\partial T} |u(x)|^2 \sigma(dx)\right)^{1/2} \cdot \sigma_n(\partial T)^{1/2}\lesssim \sigma_n(\partial T)^{1/2} \mathscr{E}_1(u,u).  
\]
This implies that $\sigma_n$ is a measure of finite energy integral and hence charges no polar sets.  As a result, $\sigma$ charges no polar sets  either. Since $\sigma$ is Radon,  we can conclude that $\sigma$ is smooth with respect to $B^T$.  
\end{proof}

Clearly, the quasi support of $\sigma$ is the closed set $\partial T$.  Set 
\[
	L^2(\partial T):=L^2(\partial T, \sigma)=\{f\text{ on }\partial T:  f_0(\cdot):=f(\cdot, 0),  f_\pi(\cdot):=f(\cdot, \pi)\in L^2(\bR)\}.
\] 
Denote the trace Dirichlet form of $(\mathscr{E}, \mathscr{F})$ on $\partial T$ with respect to $\sigma$ by $(\check{\mathscr{E}},\check{\mathscr{F}})$,  i.e.
\[
\begin{aligned}
	\check{\sF}&= \sF_e|_{\partial T} \cap L^2(\partial T),\\\
	\check{\sE}(\varphi, \phi)&=\sE(\bH \varphi,\bH \phi),\quad \forall \varphi, \phi\in \check{\sF}, 
\end{aligned}
\]
where $\sF_e=H^1_e(\mathring{T})=\{u\in L^2_\mathrm{loc}(\mathring{T}): \frac{\partial u}{\partial x_i}\in L^2(\mathring{T}), i=1,2\}$ is the extended Dirichlet space of $(\sE,\sF)$ (see, e.g.,  \cite[Theorem~2.2.13]{CF12}) and
$$
\bH \phi(x):=\mathbf{E}_x\left[\phi(B^T_{\tau}), \tau<\infty\right],\quad  \tau:=\{t>0:B^T_t\notin \mathring{T}\}. 
$$
Note that $(\check{\mathscr{E}},\check{\mathscr{F}})$ is a regular Dirichlet form on $L^2(\partial T)$ associated with the Markov process
\[
	\check{B}^T_t:=B^T_{\varsigma_t},
\]
where $\varsigma_t:=\{s>0: L_t>s\}$ is the right inverse of $L_t$,  the positive continuous additive functional (PCAF in abbreviation) corresponding to $\sigma$ with respect to $B^T$;  see, e.g.,  \cite[Theorem~6.2.1]{FOT11}.  

\section{Poisson kernel and Feller measure}

To formulate the concrete expression of $(\check{\sE},\check{\sF})$,  the crucial step is to obtain the so-called Feller measure $U(d\xi d\xi')$ on $\partial T\times \partial T\setminus d$, where $d$ is the diagonal of $\partial T\times \partial T$.  More precisely,  for $\varphi, \phi\in b\mathcal{B}_+(\partial T)$,  
\begin{equation}\label{eq:1}
	U(\varphi\otimes \phi)=\lim_{\alpha\uparrow \infty} U_\alpha(\varphi \otimes \phi),
\end{equation}
where $U_\alpha(\varphi\otimes \phi):=\alpha \int_{\mathring{T}} \bH^\alpha \varphi(x)\bH \phi(x)dx$ with $\bH^\alpha \varphi(x):=\mathbf{E}_x\left[e^{-\alpha \tau}\varphi(B^T_{\tau}), \tau<\infty\right]$; see, e.g., \cite{FHY04}.  

When the trace is considered on a compact hypersurface $\partial D$ of a (bounded) domain $D$ instead,  it has shown in,  e.g.,  \cite[Theorem~A.3.2]{AS82} and \cite[Example~2.1]{FHY04},  that the Feller measure $U_{\partial D}(d\xi d\xi')$ admits a density $U_{\partial D}(\xi,\xi')$ with respect to $\sigma_{\partial D}(d\xi)\sigma_{\partial D}(d\xi')$ where $\sigma_{\partial D}$ is the surface measure on $\partial D$,  namely $U_{\partial D}(d\xi d\xi')=U_{\partial D}(\xi,\xi')\sigma_{\partial D}(d\xi)\sigma_{\partial D}(d\xi')$,  and
\begin{equation}\label{eq:21}
U_{\partial D}(\xi,\xi')=\frac{1}{2}\frac{\partial P_D(\xi,\xi')}{\partial \mathbf{n}_\xi}, \quad \xi,\xi'\in \partial D,
\end{equation}
where $P_D$ is the Poisson kernel for $D$ and $\mathbf{n}_\xi$ denotes the inward normal unit vector at $\xi$. 

In this section,  we will figure out the Feller measure for the trace on the unbounded surface $\partial T$.  The result below indicates that  \eqref{eq:21} still holds for $T$. Note that the Poisson kernel for $T$ is
\begin{equation}\label{eq:C6}
P(x,\xi)=\frac{1}{2\pi}\frac{\sin x_2}{\cosh(x_1-\xi_1)-\cos x_2}1_{\{\dagger=0\}}+\frac{1}{2\pi}\frac{\sin x_2}{\cosh(x_1-\xi_1)+\cos x_2}1_{\{\dagger=\pi\}}
\end{equation}
for $x=(x_1,x_2)\in \mathring{T},  \xi=(\xi_1, \dagger)\in \partial T$;  see, e.g.,  \cite{W61}.  As a byproduct,  a probabilistic approach to formulate this Poisson kernel is presented in this proof. 

\begin{theorem}\label{THM1}
The Feller measure $U(d\xi d\xi')$ admits a density function $U(\xi, \xi')$ with respect to $\sigma(d\xi)\sigma(d\xi')$,  i.e.  $U(d\xi d\xi')=U(\xi,\xi')\sigma(d\xi)\sigma(d\xi')$,  and 
\begin{equation}\label{eq:4}
U(\xi,\xi')=\frac{1}{2}\frac{\partial P(\xi,\xi')}{\partial \mathbf{n}_\xi}, \quad \xi,\xi'\in \partial T,
\end{equation}
where $P$ is the Poisson kernel for $T$ given by \eqref{eq:C6}
and $\mathbf{n}_\xi$ denotes the inward normal unit vector at $\xi$.   More precisely,   for $\xi=(\xi_1,\dagger),  \xi'=(\xi'_1,\dagger')\in \partial T$,  
\begin{equation}\label{eq:5}
U(\xi,\xi')=\frac{1}{4\pi(\cosh(\xi_1-\xi_1')-1)}1_{\{\dagger=\dagger'\}}+\frac{1}{4\pi(\cosh(\xi_1-\xi_1')+1)}1_{\{\dagger\neq \dagger'\}}.
\end{equation}
\end{theorem}
\begin{proof}
Let $B^T_t=(\beta^1_t, \beta^2_t)$,  where $\beta^1$ is a one-dimensional Brownian motion and $\beta^2$ is a reflecting Brownian motion on $[0,\pi]$ independent of $\beta^1$.  Then 
\[
	\tau=\{t>0: \beta^2_t\notin (0,\pi)\}
\]
is also independent of $\beta^1$.  In addition,  for any $\varphi\in b\mathcal{B}_+(\partial T)$ and $x=(x_1,x_2)\in \mathring{T}$, 
\begin{equation}\label{eq:C11}
	\bH^\alpha \varphi(x)=\mathbf{E}_x\left[e^{-\alpha \tau}\varphi_0(\beta^1_{\tau}), \beta^2_\tau=0\right]+\mathbf{E}_x\left[e^{-\alpha \tau}\varphi_\pi(\beta^1_{\tau}), \beta^2_\tau=\pi\right]
\end{equation}
and
 \[
	\bH \varphi(x)=\mathbf{E}_x\left[\varphi_0(\beta^1_{\tau}), \beta^2_\tau=0\right]+\mathbf{E}_x\left[\varphi_\pi(\beta^1_{\tau}), \beta^2_\tau=\pi\right].
\]
Let $\beta^{2,\tau}$ be the killed Brownian motion at time $\tau$, i.e. $\beta^{2,\tau}_t:=\beta^2_t$ for $t<\tau$ and $\beta^{2,\tau}_t:=\partial$,  the trap,  for $t\geq \tau$.  Denote the probability transition density of $\beta^{2,\tau}$ with respect to the Lebesgue measure by  $p^0_t(x_2,x'_2)$.   More precisely,  we have (see, e.g., \cite[Chapter 2,  Proposition~8.2]{PS78}) for $x_2,x'_2\in (0,\pi)$,  
\begin{equation}\label{eq:C23}
	p^0_t(x_2,x'_2)=\frac{2}{\pi}\sum_{n= 1}^\infty e^{-\frac{n^2 t}{2}} \sin(nx_2)\sin(nx'_2).
\end{equation}
Define
\begin{equation}\label{eq:C21}
	h(t,x_2,0):=\lim_{x'_2\downarrow 0}\frac{1}{2}\frac{\partial p^0_t(x_2,x'_2)}{\partial x'_2}=\frac{1}{\pi}\sum_{n= 1}^\infty ne^{-\frac{n^2 t}{2}} \sin(nx_2)
\end{equation}
and
\begin{equation}\label{eq:C22}
	h(t,x_2,\pi):=-\lim_{x'_2\uparrow \pi}\frac{1}{2}\frac{\partial p^0_t(x_2,x'_2)}{\partial x'_2}=\frac{1}{\pi}\sum_{n= 1}^\infty (-1)^{n+1} ne^{-\frac{n^2 t}{2}} \sin(nx_2).
\end{equation}
In what follows we will complete the proof in several steps.

Firstly,  we assert that for $x_2\in (0,\pi)$ and $\dagger=0,\pi$, 
\begin{equation}\label{eq:C2}
	\mathbf{P}_{x_2}(\tau\in dt, \beta^2_\tau =\dagger)=h(t,x_2, \dagger)dt. 
\end{equation}
To do this,  take a positive function $g$ defined on $\{0,\pi\}$ and set $u(x_2):=\mathbf{E}_{x_2}g(\beta^2_\tau)$ for $x_2\in (0,\pi)$, which is clearly harmonic in $(0,\pi)$.    Since $p^0$ satisfies $\partial p_t^0(x_2,x'_2)=\frac{1}{2}\Delta_{x'_2}p_t^0(x_2,x'_2)$ and $p^0_t(x_2,\dagger)=0$ for $\dagger=0,\pi$,   it follows from the strong Markovian property of $\beta^2$ that
\[
\begin{aligned}
\mathbf{E}_{x_2}\left[g(\beta^2_\tau),\tau\leq t\right]&=u(x_2)-\mathbf{E}_{x_2}u(\beta^{2,\tau}_t)  \\
&=-\int_0^t \frac{\partial}{\partial s}\left(\int_0^\pi p^0_s(x_2,x'_2)u(x'_2)dx'_2\right)ds \\
&=- \frac{1}{2}\int_0^t ds \int_0^\pi \Delta_{x'_2} p^0_s(x_2,x'_2)u(x'_2)dx'_2  \\
&=\int_0^t \left(g(\pi)h(s,x_2,\pi)+g(0)h(s,x_2,0) \right) ds.
\end{aligned}\]
Taking $g(0)=1, g(\pi)=0$ or $g(0)=0, g(\pi)=1$,  we can obtain \eqref{eq:C2}.  

Secondly,  $\bH^\alpha \varphi(x)$ and $\bH \varphi(x)$ admit the following expression: For $x=(x_1,x_2)\in \mathring{T}$ and $\xi=(\xi_1,\dagger)\in \partial T$, 
\begin{equation}\label{eq:C5}
	\bH^\alpha \varphi(x)=\int_{\partial T} P_\alpha(x,  \xi) \varphi(\xi)\sigma(d\xi), \quad 
	\bH \varphi(x)=\int_{\partial T} P(x, \xi) \varphi(\xi)\sigma(d\xi),
\end{equation}
where 
\begin{equation}\label{eq:C7}
P_\alpha(x,  \xi)=\int_0^\infty e^{-\alpha t} r_t(x_1,\xi_1)h(t,x_2,\dagger)dt,\quad P(x,  \xi)=\int_0^\infty  r_t(x_1,\xi_1)h(t,x_2,\dagger)dt,
\end{equation}
and $r_t(x_1,\xi_1)=\frac{1}{\sqrt{2\pi t}}e^{-\frac{(x_1-\xi_1)^2}{2t}}$.  Particularly,  $P$ is the Poisson kernel for $T$,  and a computation by means of \eqref{eq:C21} and \eqref{eq:C22} yields \eqref{eq:C6}. 
To prove \eqref{eq:C5},  it follows from \eqref{eq:C11} and \eqref{eq:C2} that 
\[
\begin{aligned}
	\bH^\alpha \varphi(x)&=\sum_{\dagger=0,\pi} \mathbf{E}_{x_2} \left[ e^{-\alpha \tau} \cdot \mathbf{E}_{x_1}\varphi_\dagger(\beta^1_\tau),  \beta^2_\tau=\dagger  \right]  \\
	&=\sum_{\dagger=0,\pi} \mathbf{E}_{x_2} \left[ e^{-\alpha \tau} \cdot \int_\bR r_\tau(x_1,\xi_1)\varphi_\dagger(\xi_1)d\xi_1,  \beta^2_\tau=\dagger  \right]  \\
	&=\sum_{\dagger=0,\pi} \int_0^\infty e^{-\alpha t}  \left(\int_\bR r_t(x_1,\xi_1)\varphi_\dagger(\xi_1)d\xi_1\right)h(t,x_2,\dagger)dt,
\end{aligned}\]
which implies the above expression of $\bH^\alpha \varphi$.  The expression of $\bH \varphi$ can be obtained analogously.

Finally,  we turn to prove \eqref{eq:4}.  To accomplish this,  take two positive bounded functions $\varphi,\phi$ on $\partial T$. It follows from \eqref{eq:C5} that $U_\alpha$ admits the following density with respect to $\sigma(d\xi)\sigma(d\xi')$: For $\xi=(\xi_1,\dagger)$ and $\xi'=(\xi'_1, \dagger')$ with $\dagger, \dagger'\in \{0,\pi\}$,  
\[
\begin{aligned}
	U_\alpha(\xi,\xi')&=\int_{\mathring{T}}\alpha P_\alpha(x,\xi)P(x,\xi')dx\\
	&=\frac{\alpha}{4}\int_0^\infty \int_0^\infty e^{-\alpha s} r_{t+s}(\xi_1,\xi'_1)dsdt\int_0^\pi \lim_{y_2\rightarrow \dagger, z_2\rightarrow \dagger'}\frac{\partial p^0_s(x_2,y_2)}{\partial y_2}\frac{\partial p^0_t(x_2,z_2)}{\partial z_2}dx_2.
\end{aligned}\]
Since $p^0_{t+s}(y_2,z_2)=\int_0^\pi p^0_s(x_2,y_2)p^0_t(x_2,z_2)dx_2$ and $p^0_t$ enjoys the expression  \eqref{eq:C23},  it follows from the dominated convergence theorem and the Fubini theorem that
\[
\begin{aligned}
	U_\alpha(\xi,\xi')&=\frac{\alpha}{4}\int_0^\infty  \int_0^\infty e^{-\alpha s} r_{t+s}(\xi_1,\xi'_1) \lim_{y_2\rightarrow \dagger, z_2\rightarrow \dagger'}\frac{\partial^2 p^0_{t+s}(y_2,z_2)}{\partial y_2 \partial z_2}dsdt  \\
	&=\frac{1}{4}\int_0^\infty (1-e^{-\alpha t}) r_t(\xi_1,\xi'_1) \lim_{y_2\rightarrow \dagger, z_2\rightarrow \dagger'}\frac{\partial^2 p^0_{t+s}(y_2,z_2)}{\partial y_2 \partial z_2}dt.
\end{aligned}\]
It is easy to verify that $t\mapsto r_t(\xi_1,\xi'_1) \lim_{y_2\rightarrow \dagger, z_2\rightarrow \dagger'}\frac{\partial^2 p^0_{t+s}(y_2,z_2)}{\partial y_2 \partial z_2}$ is integrable.  Hence,  using \cite[(5.5.14)]{CF12},  we can obtain that
\begin{equation}\label{eq:C112}
	U(\xi,\xi')=\lim_{\alpha\uparrow \infty}  U_\alpha(\xi,\xi')=\frac{1}{4}\int_0^\infty  r_t(\xi_1,\xi'_1) \lim_{y_2\rightarrow \dagger, z_2\rightarrow \dagger'}\frac{\partial^2 p^0_{t+s}(y_2,z_2)}{\partial y_2 \partial z_2}dt.
\end{equation}
Eventually we can obtain \eqref{eq:4} on account of \eqref{eq:C7} and \eqref{eq:C112}.  Particularly a straightforward computation yields \eqref{eq:5}.  That completes the proof. 
\end{proof}

\section{Characterization of trace Dirichlet form}

The purpose of this section is to derive the concrete expression of the trace Dirichlet form $(\check{\sE}, \check{\sF})$.  
The main result is as follows.

\begin{theorem}\label{THM2}
The trace Dirichlet form $(\check{\mathscr{E}},\check{\mathscr{F}})$ on $L^2(\partial T)$ enjoys the following expression:
\begin{equation}\label{eq:C1}
\begin{aligned}
\check{\mathscr{F}}&=H^{1/2}(\partial T):=\{f\in L^2(\partial T): f_0,f_\pi \in H^{1/2}(\bR)\},\\
\check{\mathscr{E}}(f,f)
& =\frac{1}{4\pi}\int_{\mathbb{R}\times\mathbb{R}}\frac{\left(f_0(x_1)-f_\pi(x_1')\right)^2}{\cosh(x_1-x_1')+1}dx_1dx_1'\\
&\qquad \qquad+\frac{1}{8\pi}\int_{\mathbb{R}\times\mathbb{R}}\frac{\left(f_0(x_1)-f_0(x_1')\right)^2+\left(f_\pi(x_1)-f_\pi(x_1')\right)^2}{\cosh(x_1-x_1')-1}dx_1dx_1',\quad f\in \check{\mathscr{F}}. 
\end{aligned}
\end{equation}
Furthermore,  the norms $\|\cdot\|_{\check{\mathscr{E}}_1}$ and $\|\cdot\|_{H^{1/2}(\partial T)}$ on $\check{\mathscr{F}}$ are equivalent.  
\end{theorem}
\begin{proof}
Firstly, note that $B^T$ is recurrent due to \cite[Theorem~2.2.13]{CF12}.  Then it follows from \cite[Theorem~5.2.5]{CF12} that $\check{B}^T$ is recurrent.  Particularly,  $\check{B}^T$ has no killing inside.  Furthermore, by means of \cite[Corollary~5.6.1]{CF12}), we have that for $f\in \check{\mathscr{F}}$,
\[
	\check{\mathscr{E}}(f,f)=\frac{1}{2}\mu^c_{\langle \bH f, \bH f\rangle}(\partial T) +\frac{1}{2}\int_{\partial T\times \partial T} \left(f(\xi)-f(\xi')\right)^2 U(d\xi d\xi'),
\]
where $\mu^c_{\langle \bH f, \bH f\rangle}$ is the smooth measure expressing the strongly local part of $\mathscr{E}(\bH f, \bH f)$ and $U(d\xi d\xi')$ is the Feller measure formulated in Theorem~\ref{THM1}.  It is easy to verify that $\mu^c_{\langle \bH f, \bH f\rangle}(dx)=|\nabla (\bH f)|^2(x)dx$ and hence $\mu^c_{\langle \bH f, \bH f\rangle}(\partial T)=0$.  
As a consequence of \eqref{eq:5},  $\check{\mathscr{E}}$ enjoys the expression in \eqref{eq:C1}.

Secondly,  we show 
$$\check{\mathscr{F}}=\{f\in L^2(\partial T): \check{\mathscr{E}}(f,f)<\infty\}.$$ 
In fact,  $\check{\mathscr{F}}=\check{\mathscr{F}}_e\cap L^2(\partial T)$ and \cite[Example~7.2.10]{CF12} tells us that 
\begin{equation}\label{eq:C12}
	\check{\mathscr{F}}_e=\{f\text{ on }\partial T: \bH(f^2)<\infty,  \text{ a.e.  on }\mathring{T}\text{ and } \check{\mathscr{E}}(f,f)<\infty\}.  
\end{equation}
It suffices to show $\bH(f^2)<\infty$ a.e. on $\mathring{T}$ for $f\in L^2(\partial T)$ with $\check{\mathscr{E}}(f,f)<\infty$.  Note that $\bH (f^2)$ is harmonic on $\mathring{T}$ due to \cite[Theorem~1]{W61}.  On the other hand, since $\cosh \xi_1\geq 1$, 
\begin{equation}\label{eq:C13}
	\bH(f^2)(0,\pi/2)=\frac{1}{2\pi}\sum_{\dagger=0,\pi}\int_\bR \frac{f^2_\dagger(\xi_1)}{\cosh \xi_1} d\xi_1\leq \frac{1}{2\pi}\int_{\partial T} f^2d\sigma<\infty .  
\end{equation}
Therefore $\bH(f^2)<\infty$ on $\mathring{T}$ by virtue of the Harnack inequality.  

Thirdly,  denote the trace operator from $H^1(T)$ to $H^{1/2}(\partial T)$ by $\gamma^T$,  i.e. 
\begin{equation}\label{eq:gamma}
\gamma^T: H^1(T)\rightarrow H^{1/2}(\partial T)
\end{equation}
is a bounded linear operator such that $\gamma^T u=u|_{\partial T}$ for any $u\in C_c^\infty(T)$.  Note that $\cosh\xi_1-1\geq \xi_1^2/2$ and $\xi_1\mapsto \frac{1}{\cosh \xi_1 +1}$ is integrable.  These imply that 
\begin{equation}\label{eq:C132}
	H^{1/2}(\partial T)\subset \check{\mathscr{F}},\quad \check{\mathscr{E}}_1(\varphi,\varphi)\lesssim \|\varphi\|^2_{H^{1/2}(\partial T)},\quad \forall \varphi\in H^{1/2}(\partial T).
\end{equation}
We assert that for any $f\in H^1(T)$,  it holds that
\begin{equation}\label{eq:C14}
	\tilde{f}|_{\partial T}=\widetilde{\gamma^T f},\quad \check{\mathscr{E}}\text{-q.e.},
\end{equation}
where $\tilde{f}$ is the $\mathscr{E}$-quasi-continuous version of $f$ and $\widetilde{\gamma_T f}$ is the $\check{\mathscr{E}}$-quasi-continuous version of $\gamma_T f\in H^{1/2}(\partial T)\subset \check{\mathscr{F}}$.  To accomplish this,  take $f_n\in C_c^\infty(T)$ such that $\mathscr{E}_1(f_n-f,f_n-f)\rightarrow 0$.  Then we have,  taking a subsequence of $\{f_n\}$ if necessary,  $f_n$ converges to $\tilde{f}$,  $\mathscr{E}$-q.e.  This implies $f_n|_{\partial T}$ converges to $\tilde{f}|_{\partial T}$, $\mathscr{E}$-q.e.  It follows from \cite[Theorem~5.2.8]{CF12} that $f_n|_{\partial T}$ converges to $\tilde{f}|_{\partial T}$,  $\check{\mathscr{E}}$-q.e.  On the other hand,  $\gamma^T f_n\rightarrow \gamma^T f$ in $H^{1/2}(\partial T)$.   Hence \eqref{eq:C132} indicates that, taking a subsequence of $\{f_n\}$ if necessary, $f_n|_{\partial T}=\gamma^T f_n$ converges to $\widetilde{\gamma^T f}$,  $\check{\mathscr{E}}$-q.e.  Therefore we can conclude \eqref{eq:C14}. 

Fourthly,  we assert that 
\begin{equation}\label{eq:C15}
f(x):=\mathbf{H} \varphi(x) =\int_{\partial T} P(x,\xi)\varphi(\xi)\sigma(d\xi)\in \mathscr{F},\quad \forall \varphi \in \check{\mathscr{F}},
\end{equation}
 where $P(x,\xi)$ is the Poisson kernel \eqref{eq:C6} for $T$.  Clearly $f\in \mathscr{F}_e$, and it suffices to show $f\in L^2(T)$. Indeed, for  $x=(x_1,x_2)\in T$, 
 \[
 	f(x_1,x_2)=\sum_{\dagger=0, \pi}  P^\dagger_{x_2}\ast \varphi_\dagger(x_1),
 \]
 where $P^0_{x_2}(\cdot)=\frac{1}{2\pi} \frac{\sin x_2}{\cosh(\cdot)-\cos x_2}$ and $P^\pi_{x_2}(\cdot)=\frac{1}{2\pi} \frac{\sin x_2}{\cosh(\cdot)+\cos x_2}$.  Note that 
 \[
 	\|P^0_{x_2}\|_{L^1(\bR)}=1-x_2/\pi,\quad \|P^0_{x_2}\|_{L^1(\bR)}=x_2/\pi,
 \]
 due to \cite[Lemma~1]{W61}.  It follows from the Young inequality for convolutions that 
 \[
 	\|f(\cdot, x_2)\|_{L^2(\bR)}^2 = \sum_{\dagger=0,\pi} \|P^\dagger_{x_2}\ast \varphi_\dagger\|_{L^2(\bR)}^2\leq \sum_{\dagger=0,\pi} \|P^\dagger_{x_2}\|_{L^1(\bR)}^2 \|\varphi_\dagger\|_{L^2(\bR)}^2 \leq \|\varphi\|_{L^2(\partial T)}^2.  
 \]
Hence we can conclude that 
\begin{equation}\label{eq:C16}
\|f\|_{L^2(T)}^2=\int_0^\pi \|f(\cdot, x_2)\|_{L^2(\bR)}^2dx_2\leq \pi \|\varphi\|_{L^2(\partial T)}^2.
\end{equation}

Finally,  we prove that 
\begin{equation}\label{eq:C18}
	\check{\mathscr{F}}\subset H^{1/2}(\partial T),\quad \|\varphi\|^2_{H^{1/2}(\partial T)}\lesssim \check{\mathscr{E}}_1(\varphi,\varphi),\quad \forall \varphi\in \check{\mathscr{F}}.
\end{equation}
To do this,  take $\varphi\in \check{\mathscr{F}}$ and let $f:=\bH \varphi$.  On account of \eqref{eq:C15},  we know that $f\in \mathscr{F}=H^1(T)$.  Then it follows from \eqref{eq:C14} that 
\[
\varphi=f|_{\partial T}=\gamma^T f\in H^{1/2}(\partial T). 
\]
In addition,  $\|\varphi\|^2_{H^{1/2}(\partial T)}=\|\gamma^T f\|^2_{H^{1/2}(\partial T)}\lesssim \mathscr{E}_1(f,f)\leq \check{\mathscr{E}}(\varphi,\varphi)+\pi\cdot \|\varphi\|^2_{L^2(\partial T)}$ due to \eqref{eq:C16}.
That completes the proof. 
\end{proof}
\begin{remark}
The existence of the trace operator $\gamma^T$ in \eqref{eq:gamma} is due to,  e.g.,  \cite{JW78}. 
From this proof,  we can conclude that the unitary equivalence 
\begin{equation}\label{eq:C17}
	(\mathscr{E},\mathscr{F}) \xleftrightharpoons[\gamma^T=\cdot|_{\partial T}]{\mathbf{H}}(\check{\mathscr{E}},\check{\mathscr{F}})
 \end{equation}
 holds.  More specifically,  $\gamma^T(\mathscr{F})=\check{\mathscr{F}}$ and $\mathscr{E}(\mathbf{H}\varphi, \bH \psi)=\check{\mathscr{E}}(\varphi,\psi)$ for $\varphi, \psi\in \check{\mathscr{F}}$.  We should emphasize that $\gamma^T=\cdot|_{\partial T}$ holds in the sense of \eqref{eq:C14}.
 \end{remark}
 
 The expression of $\check{\sE}$ shows that the associated Markov process $\check{B}^T$ enjoys two kinds of jump: The first kind, characterized by the first term in $\check{\sE}$,  takes place between the two components of $\partial T$,  and the second kind,  characterized by the second term in $\check{\sE}$,  takes place on each component of $\partial T$.  Thanks to the first kind of jump,  $\check{B}^T$ or $(\check{\sE},\check{\sF})$ is irreducible.  
 
 \begin{corollary}
 The trace Dirichlet form $(\check{\sE},\check{\sF})$ is irreducible and recurrent.  
 \end{corollary}
 \begin{proof}
 The recurrence of $(\check{\sE},\check{\sF})$ has been obtained in the proof of Theorem~\ref{THM2}.  To prove the irreducibility,  take $f\in \check{\sF}_e$ with $\check{\sE}(f,f)=0$.  It is easy to verify that $f$ must be constant.  Hence by virtue of \cite[Theorem~5.2.6]{CF12},  we can conclude the irreducibility of $(\check{\sE},\check{\sF})$.  That completes the proof. 
 \end{proof}

\section{Convergence of trace Dirichlet forms}

Now we turn to consider the strip $\Omega_\ell:=\bR\times (-\frac{\pi}{2}\ell,  \frac{\pi}{2}\ell)$ for a constant $\ell>0$.  Let $B^\ell$ be the reflecting Brownian motion on $\bar{\Omega}_\ell$ associated with the regular Dirichlet form on $L^2(\bar{\Omega}_\ell)$:
\[
\begin{aligned}
  \mathscr{F}^\ell &=H^1(\bar\Omega_\ell),\\
 \mathscr{E}^\ell (u,u)&=
 \frac{1}{2}\int_{\bar\Omega_\ell} |\nabla u|^2dx,\quad u\in \mathscr{F}^\ell.
 \end{aligned}
\]
After a spatial translation and scaling on $(\mathscr{E}, \mathscr{F})$, it is easy to obtain the trace Dirichlet form  of $(\mathscr{E}^\ell, \mathscr{F}^\ell)$ on the boundary $\partial \bar\Omega_\ell$ as follows.  

\begin{corollary}
Let $\sigma_\ell$ be the Lebesgue measure on $\partial \bar \Omega_\ell$.  Then the trace Dirichlet form of  $(\mathscr{E}^\ell, \mathscr{F}^\ell)$  on $L^2(\partial \bar{\Omega}_\ell):=L^2(\partial \bar\Omega_\ell, \sigma_\ell)$ is
\begin{equation}\label{TRACEDF}
\begin{split}
\check{\mathscr{F}}^\ell &=H^{1/2}(\partial \bar\Omega_\ell):=\{f\in L^2(\partial \bar{\Omega}_\ell): f_{\pm \ell}\in H^{1/2}(\bR)\},\\
\check{\mathscr{E}}^\ell(f,f)
& =\frac{1}{4\pi}\int_{\mathbb{R}\times\mathbb{R}}\frac{\left(f_\ell(x_1)-f_{-\ell}(x_1')\right)^2}{\ell^2\left(\cosh\left(\ell^{-1}(x_1-x_1')\right)+1\right)}dx_1dx_1'\\
&\qquad \qquad+\frac{1}{8\pi}\int_{\mathbb{R}\times\mathbb{R}}\frac{\left(f_\ell(x_1)-f_\ell(x_1')\right)^2+\left(f_{-\ell}(x_1)-f_{-\ell}(x_1')\right)^2}{\ell^2\left(\cosh\left(\ell^{-1}(x_1-x_1')\right)-1\right)}dx_1dx_1',\quad f\in \check{\mathscr{F}}^\ell, 
\end{split}
\end{equation}
where $f_{\ell}(\cdot):=f(\cdot,\ell)$ and $f_{-\ell}(\cdot ):=f(\cdot,-\ell)$.  Furthermore,  the norms $\|\cdot\|_{\check{\mathscr{E}}^\ell_1}$ and $\|\cdot\|_{H^{1/2}(\partial \bar\Omega_\ell)}$ on $\check{\mathscr{F}}^\ell$ are equivalent.  
\end{corollary}

The main purpose of this section is to figure out the limit of  \eqref{TRACEDF} as $\ell\rightarrow \ell_0\in [0,\infty]$.  Note that $\partial \bar{\Omega}_\ell=\bR\times \{\ell, -\ell\}$ is isomorphic to $\partial T:=\bR\times \{0,\pi\}$ under a certain map $j$.  Hence,  the image Dirichlet form of $(\check{\mathscr{E}}^\ell, \check{\mathscr{F}}^\ell)$ under $j$ is 
\[
\begin{aligned}
	\check{\mathscr{G}}^\ell &=H^{1/2}(\partial T),   \\
	\check{\mathscr{A}}^\ell(f,f)&=\frac{1}{2\pi}\check{\mathscr{A}}^{\ell,1}(f,f)+\frac{1}{8\pi }\check{\sA}^{\ell,2}(f,f),\quad f\in \check{\mathscr{G}}^\ell,
\end{aligned} 
\]
where $\partial T^+:=\bR\times \{\pi\},  \partial T^-:=\bR\times \{0\}$, $f^+(\cdot):=f(\cdot, \pi)$, $f^-(\cdot):=f(\cdot, 0)$,  and
\[
\begin{aligned}
&\check{\mathscr{A}}^{\ell,1}(f,f):=\int_{\mathbb{R}\times\mathbb{R}}\frac{\left(f^+(x_1)-f^-(x_1')\right)^2}{2\ell^2\left(\cosh\left(\ell^{-1}(x_1-x_1')\right)+1\right)}dx_1dx_1', \\
&\check{\sA}^{\ell,2}(f,f):=\int_{\mathbb{R}\times\mathbb{R}}\frac{\left(f^+(x_1)-f^+(x_1')\right)^2+\left(f^-(x_1)-f^-(x_1')\right)^2}{\ell^2\left(\cosh\left(\ell^{-1}(x_1-x_1')\right)-1\right)}dx_1dx_1'.
\end{aligned}
\]
Clearly,  $(\check{\sA}^\ell, \check{\sG}^\ell)$ is a regular Dirichlet form on $L^2(\partial T)$.  
The following lemma states some crucial facts about the two terms $\check{\sA}^{\ell,1}$ and $\check{\sA}^{\ell,2}$.  

\begin{lemma}\label{LMD1}
\begin{itemize}
\item[(1)] For any $f\in L^2(\partial T)$,  it holds that $\check{\mathscr{A}}^{\ell,1}(f,f)\leq \frac{2}{\ell}\|f\|^2_{L^2(\partial T)}$. 
\item[(2)] For any $f\in H^{1/2}(\partial T)$,  
\[
	\ell \rightarrow \check{\mathscr{A}}^{\ell,2}(f,f)
\]
is increasing.  Furthermore,  $\|\cdot\|_{{\check{\sA}^{\ell,2}}_1}$ is an equivalent norm to $\|\cdot\|_{H^{1/2}(\partial T)}$ on $H^{1/2}(\partial T)$.  Particularly,  $(\check{\sA}^{\ell,2}, \check{\mathscr{G}}^\ell)$ is a regular Dirichlet form on $L^2(\partial T)$.   
\end{itemize}
\end{lemma}
\begin{proof}
\begin{itemize}
\item[(1)] Note that $\int_{\mathbb{R}}\phi(x_1)dx_1=1$ for $\phi(x_1):=\frac{1}{2(\cosh x_1+1)}$.  Hence
\begin{equation}\label{eq:D1}
	\sA^{\ell,1}(f,f)\leq \frac{2}{\ell}\int_{\mathbb{R}\times\mathbb{R}}\frac{\left(f^+(x_1)\right)^2+\left(f^-(x_1')\right)^2}{2\ell\left(\cosh\left(\ell^{-1}(x_1-x_1')\right)+1\right)}dx_1dx_1' =\frac{2}{\ell}\|f\|_{L^2(\partial T)}^2.
\end{equation}
\item[(2)] Note that for any fixed $x_1\in \bR$, the function 
$\ell\mapsto\ell^2\left(\cosh\left(\ell^{-1}x_1\right)-1\right)$ is decreasing on $(0,\infty)$, and 
\begin{equation}\label{EQcosh}
\ell^2\left(\cosh\left(\ell^{-1}x_1\right)-1\right)\geq\lim_{\ell\uparrow \infty}\ell^2\left(\cosh\left(\ell^{-1}x_1\right)-1\right)=\frac{x_1^2}{2}. 
\end{equation}
The first assertion is clear by means of \eqref{EQcosh}.  Particularly,  for $f\in H^{1/2}(\partial T)$,
\[
\check{\mathscr{A}}^{\ell,2}(f,f)\leq 2\int_{\mathbb{R}\times\mathbb{R}}\frac{\left(f^+(x_1)-f^+(x_1')\right)^2+\left(f^-(x_1)-f^-(x_1')\right)^2}{(x_1-x_1')^2}dx_1dx_1'\lesssim \|f\|^2_{H^{1/2}(\partial T)}.  
\]
Then the equivalence between $\|\cdot\|_{{\check{\sA}^{\ell,2}}_1}$ and $\|\cdot\|_{H^{1/2}(\partial T)}$ is due to \eqref{eq:C18} and \eqref{eq:D1}.  
\end{itemize}
\end{proof}
\begin{remark}
From the first assertion,  one may also find that $(\check{\sA}^{\ell,1}, L^2(\partial T))$ and $(\ell\cdot \check{\sA}^{\ell,1}, L^2(\partial T))$ are both regular Dirichlet forms on $L^2(\partial T)$.  
\end{remark}

In what follows,  we will derive the limit of $(\check{\mathscr{A}}^\ell, \check{\mathscr{G}}^\ell)$ in place of $(\check{\sE}^\ell, \check{\sF}^\ell)$ as $\ell\rightarrow \ell_0$.  To do this,  the concept of so-called Mosco convergence will be employed;  see Appendix~\ref{APP}.  Note that the Mosco convergence implies the convergence of finite dimensional distributions of associated Markov processes; see, e.g.,  \cite[Corollary~4.1]{LS19}. Before stating the main result,  we first figure out the limit of $\check{\mathscr{A}}^{\ell,i}$ for $i=1,2$.  The following result is concerned with the case $i=1$.  Particularly,  at a heuristic level,  the divergence of $\check{\mathscr{A}}^{\ell,1}$ as $\ell\downarrow 0$ is concluded. 

\begin{proposition}\label{PROD3}
Let $\ell_n$ be a sequence in $(0,\infty)$ such that $\lim_{n\rightarrow\infty}\ell_n=\ell_0\in[0,\infty]$.  Then the following hold:
\begin{itemize}
\item[(1)] When $\ell_0=0$,  it holds for every $f\in L^2(\partial T)$ that
\begin{equation}\label{eq:D2-2}
		\lim_{n\rightarrow \infty}\ell_n\cdot \check{\mathscr{A}}^{\ell_n,1}(f,f)=\int_\mathbb{R}\left(f^+(x_1)-f^-(x_1)\right)^2dx_1=:\check{\sA}^{0,1}(f,f). 
\end{equation}
Furthermore,  $(\ell_n\cdot \check{\sA}^{\ell_n,1}, L^2(\partial T))$ converges to $(\check{\sA}^{0,1}, L^2(\partial T))$ in the sense of Mosco as $\ell_n\rightarrow 0$.  
\item[(2)] When $\ell_0\in (0,\infty]$,  it holds for every $f\in L^2(\partial T)$ that 
\begin{equation}\label{eq:D3-3}
	\lim_{n\rightarrow \infty} \check{\sA}^{\ell_n,1}(f,f)=\check{\sA}^{\ell_0,1}(f,f),
\end{equation}
where $\check{\sA}^{\infty,1}(f,f):=0$.  Furthermore,  $(\check{\sA}^{\ell_n,1}, L^2(\partial T))$ converges to $(\check{\sA}^{\ell_0,1}, L^2(\partial T))$ in the sense of Mosco as $\ell_n\rightarrow \ell_0$. 
\end{itemize}
\end{proposition}
\begin{proof}
\begin{itemize}
\item[(1)] Let $\phi$ be in the proof of Lemma~\ref{LMD1}.  Note that $\phi_n(x_1):=\ell_n^{-1}\phi(x_1/\ell_n)$ forms a class of approximations to identity.  Hence
\[
((f^-)*\phi_n)(x_1)=\int_\mathbb{R}\frac{f^-(x_1')}{2\ell_n\left(\cosh\left(\ell_n^{-1}(x_1-x_1')\right)+1\right)}dx_1'\rightarrow f^-(x_1)
\]
in $L^2(\mathbb{R})$. It follows that 
\[\int_{\mathbb{R}\times\mathbb{R}}\frac{f^+(x_1)f^-(x_1')}{2\ell_n\left(\cosh\left(\ell_n^{-1}(x_1-x_1')\right)+1\right)}dx_1dx_1'\rightarrow\int_\mathbb{R}f^+(x_1)f^-(x_1)dx_1.\]
Together with \eqref{eq:D1},  we can obtain  \eqref{eq:D2-2}.  To prove the Mosco convergence, it suffices to verify its first part Mosco (a) due to   \eqref{eq:D2-2}.  To do this,  set for every $f,g\in L^2(\partial T)$,
\[
	(f,g)_{m_{n1}}:=\ell_n\cdot \check{\sA}^{\ell_n,1}(f,g),\quad (f,g)_{m_{01}}:=\check{\sA}^{0,1}(f,g), 
\]
and take a sequence $\{f_n\}\subset L^2(\partial T)$ such that $f_n$ converges to $f$ weakly in $L^2(\partial T)$.  We only need to show 
\begin{equation}\label{eq:D3-2}
	(f_n-f,f)_{m_{n1}}\rightarrow 0,\quad n\rightarrow\infty,
\end{equation}
so that
\[
(f,f)_{m_{01}}=\lim_{n\rightarrow\infty}(f,f)_{m_{n1}}=\lim_{n\rightarrow\infty}(f_n,f)_{m_{n1}}\leq \lim_{n\rightarrow\infty} (f_n,f_n)_{m_{n1}}^{\frac{1}{2}}(f,f)_{m_{01}}^{\frac{1}{2}}
\]
leads to $(f,f)_{m_{01}}\leq  \lim_{n\rightarrow\infty} (f_n,f_n)_{m_{n1}}$.  In fact, 
\[
\begin{split}
 \qquad (f_n-f,f)_{m_{n1}}&=\int_{\mathbb{R}\times\mathbb{R}}\frac{\left((f^+_n-f^+)(x_1)-(f^-_n-f^-)(x_1')\right)(f^+(x_1)-f^-(x_1'))}{2\ell_n\left(\cosh\left(\ell_n^{-1}(x_1-x_1')\right)+1\right)}dx_1dx_1'\\
&=: I_1^++I_1^-+I_2^++I_2^-,
\end{split}
\] 
where 
\[
I_1^\pm=\int_{\mathbb{R}\times\mathbb{R}}\frac{(f^\pm_n- f^\pm)(x_1) f^\pm(x_1)}{2\ell_n\left(\cosh\left(\ell_n^{-1}(x_1-x_1')\right)+1\right)}dx_1dx_1'
\]
and 
\begin{equation}\label{eq:D5}
I_2^\pm=-\int_{\mathbb{R}\times\mathbb{R}}\frac{(f^\pm_n- f^\pm)(x_1) f^\mp(x_1')}{2\ell_n\left(\cosh\left(\ell_n^{-1}(x_1-x_1')\right)+1\right)}dx_1dx_1'.
\end{equation}
For the terms $I_1^\pm$,  it follows from $\int_{\mathbb{R}}\phi_n(x)dx=1$ and
$f_n^\pm\rightarrow  f^\pm$ weakly in $L^2(\bR)$ that
\[
I_1^\pm= \int_{\mathbb{R}}(f^\pm_n-f^\pm)(x_1)f^\pm(x_1)dx_1\rightarrow 0.
\]
For the terms $I_2^\pm$,  it follows from $g_{n}^{\mp}:= f^\mp *\phi_n\rightarrow f^\mp$ in $L^2(\mathbb{R})$ and $f^\pm_n\rightarrow f^\pm$ weakly in $L^2(\mathbb{R})$ that $I_2^\pm\rightarrow 0$.  Therefore \eqref{eq:D3-2} can be concluded.
\item[(2)]  When $\ell_0\in (0,\infty)$,  note that $\ell\mapsto \cosh \left( \ell^{-1}(x_1-x'_1)\right)$ is decreasing and \eqref{eq:D3-3} is implied by the monotone convergence theorem.  When $\ell_0=\infty$,  it follows from Lemma~\ref{LMD1}~(1) that
\[
\check{\sA}^{\ell_n,1}(f,f) \leq  \frac{2}{\ell_n}\|f\|_{L^2(\partial T)}^2\rightarrow 0.
\]
Hence \eqref{eq:D3-3} is obtained.  The second part (b) of Mosco convergence is clear due to \eqref{eq:D3-3}.  When $\ell_0=\infty$,  the first part (a) of Mosco convergence trivially holds.  
 To prove the first part (a) of Mosco convergence for the case $\ell_0\in (0,\infty)$,  it suffices to show $I^\pm_2$ in \eqref{eq:D5} converge to $0$ as $n\rightarrow \infty$.  Indeed,  it follows from the Young inequality for convolutions that 
\[
\|g_{n}^\mp-g_0^\mp\|_{L^2(\mathbb{R})}\lesssim \|f^\mp\|_{L^2(\mathbb{R})}\int_\mathbb{R}\left|\frac{1}{\cosh \ell_n^{-1}x_1'+1}-\frac{1}{\cosh \ell_0^{-1}x_1'+1}\right|dx_1',
\]
where $g_0^\mp:=f^\mp*\phi_0$ for $\phi_0(x):=\ell_0^{-1}\phi(x/\ell_0)$.
The dominated convergence theorem yields
\[
\int_\mathbb{R}\left|\frac{1}{\cosh \ell_n^{-1}x_1'+1}-\frac{1}{\cosh \ell_0^{-1}x_1'+1}\right|dx_1'\rightarrow 0.
\]
Hence $g_n^\mp\rightarrow g_0^\mp$ in $L^2(\mathbb{R})$.   From $f^\pm_n\rightarrow f^\pm$ weakly in $L^2(\mathbb{R})$,  we can eventually conclude that $I_2^\pm\rightarrow 0$.  
\end{itemize}
\end{proof}

Next,  we state the result concerning the limit of $\check{\sA}^{\ell,2}$ as $\ell\rightarrow \ell_0\in [0,\infty]$.  To do this,  set $\check{\sA}^{0,2}(f,f):=0$ for any $f\in \check{\mathscr{G}}^{0}:=L^2(\partial T)$, and for any $f\in \check{\sG}^\infty:= H^{1/2}(\partial T)$,
\[
\quad \check{\sA}^{\infty,2}(f,f):=2\int_{\mathbb{R}\times\mathbb{R}}\frac{\left(f^+(x_1)-f^+(x_1')\right)^2+\left(f^-(x_1)-f^-(x_1')\right)^2}{(x_1-x_1')^2}dx_1dx_1'.
\]
It is easy to see that the Dirichlet form $(\check{\sA}^{\infty,2}, \check{\sG}^\infty)$ is associated with a Markov process,  which consists of two distinct Cauchy processes, i.e. $\alpha$-stable processes with $\alpha=1$, on $\partial T^\pm$ respectively. 

\begin{proposition}\label{PROD4}
Let $\ell_n$ be a sequence in $(0,\infty)$ such that $\lim_{n\rightarrow\infty}\ell_n=\ell_0\in[0,\infty]$.  It holds for every $f\in H^{1/2}(\partial T)$ that
\begin{equation}\label{eq:D6}
	\lim_{n\rightarrow \infty} \check{\sA}^{\ell_n,2}(f,f)=\check{\sA}^{\ell_0,2}(f,f).
\end{equation}
Furthermore,  $(\check{\sA}^{\ell_n,2},  \check{\mathscr{G}}^{\ell_n})$ converges to $(\check{\sA}^{\ell_0,2},\check{\mathscr{G}}^{\ell_0})$ in the sense of Mosco as $n\rightarrow \infty$. 
\end{proposition}
\begin{proof}
Clearly,  \eqref{eq:D6} is a consequence of \eqref{EQcosh} and the dominated convergence theorem.  In what follows we prove the Mosco convergence for the three cases respectively.

For the case $\ell_0=0$ the first part (a) of Mosco convergence trivially holds. To prove the second part (b),  let $f\in L^2(\partial T)$.  Take a sequence $\{g_j:j\geq 1\}\subset H^{1/2}(\partial T)$ such that $g_j$ converges to $f$ in $L^2(\partial T)$.  Since for each $j\geq 1$,
\[
		\lim_{n\rightarrow \infty}\check{\sA}^{\ell_n,2}(g_j,g_j)=\check{\sA}^{0,2}(g_j,g_j)=0, 
\]
we can find an increasing subsequence $\{n_j\}$ such that
\[
	\check{\sA}^{\ell_n,2}(g_j,g_j)<2^{-j},\quad \forall n\geq n_j.
\]
Put $f_1=\cdots=f_{n_1-1}=0$ and $f_{n_j}=\cdots =f_{n_{j+1}-1}=g_j$ for $j\geq 1$.  One can easily verify that $f_n$ converges to $f$ in $L^2(\partial T)$ and $\lim_{n\rightarrow \infty}\check{\sA}^{\ell_n,2}(f_n,f_n)=0$.  
Next we consider the case $\ell_0\in (0,\infty)$.  The second part (b) of Mosco convergence trivially holds due to \eqref{eq:D6}.  It suffices to prove the first part (a).  For any $\varepsilon>0$,  there exists $N>0$ such that for any $n>N$,  $\ell_n>\ell_0-\varepsilon$.  Then for any $n>N$,  
\[
	\check{\sA}^{\ell_n,2}(f,f)\geq \check{\sA}^{\ell_0-\varepsilon,2}(f,f),\quad \forall f\in \check{\mathscr{G}}^{\ell_n}=\check{\mathscr{G}}^{\ell_0-\varepsilon}.
\]
Mimicking the proof of \cite[Theorem~2.3]{B09},  we can obtain that for any sequence $\{f_n\}$ converging to $f$ weakly in $L^2(\partial T)$,  it holds 
\[
\liminf_{n\rightarrow \infty} \check{\sA}^{\ell_n,2}(f_n,f_n)=\liminf_{n>N,  n\rightarrow \infty} \check{\sA}^{\ell_n,2}(f_n,f_n)\geq \check{\sA}^{\ell_0-\varepsilon,2}(f,f).
\]
Since $\lim_{\varepsilon\downarrow 0}\check{\sA}^{\ell_0-\varepsilon,2}(f,f)=\check{\sA}^{\ell_0,2}(f,f)$ due to \eqref{eq:D6},  we eventually have
\[
	\liminf_{n\rightarrow \infty} \check{\sA}^{\ell_n,2}(f_n,f_n)\geq \check{\sA}^{\ell_0,2}(f,f).
\]
Finally, the Mosco convergence for the case $\ell_0=\infty$ can be derived analogously.  That completes the proof.
\end{proof}

Set 
\[
	\check{\sA}^0(f,f):=\frac{1}{2\pi}\check{\sA}^{0,1}(f,f),\quad f\in \check{\mathscr{G}}^0=L^2(\partial T),
\]
and 
\[
	\check{\sA}^\infty(f,f):=\frac{1}{8\pi}\check{\sA}^{\infty,2}(f,f),\quad f\in \check{\mathscr{G}}^\infty=H^{1/2}(\partial T).
\]
Now we have a position to present the limit of the trace Dirichlet forms $(\check{\sA}^\ell,\check{\mathscr{G}}^\ell)$ as $\ell\rightarrow \ell_0\in [0,\infty]$.  Particularly,  they diverge as $\ell\downarrow 0$.  

\begin{theorem}\label{THMD5}
Let $\ell_n$ be a sequence in $(0,\infty)$ such that $\lim_{n\rightarrow\infty}\ell_n=\ell_0\in[0,\infty]$.  Then the following hold:
\begin{itemize}
\item[(1)] When $\ell_0=0$,  $(\ell_n\cdot \check{\sA}^{\ell_n}, \check{\mathscr{G}}^{\ell_n})$ converges to $(\check{\sA}^0,\check{\mathscr{G}}^0)$ as $n\rightarrow \infty$ in the sense of Mosco.
\item[(2)] When $\ell_0\in (0,\infty]$,  $(\check{\sA}^{\ell_n}, \check{\mathscr{G}}^{\ell_n})$ converges to $(\check{\sA}^{\ell_0},\check{\mathscr{G}}^{\ell_0})$ as $n\rightarrow \infty$ in the sense of Mosco.
\end{itemize}
\end{theorem}
\begin{proof}
\begin{itemize}
\item[(1)] Let $\{f_n\}$ be a sequence that converges to $f$ weakly in $L^2(\partial T)$.  Then it follows from Proposition~\ref{PROD3}~(1) that 
\[
	\check{\sA}^0(f,f)=\frac{1}{2\pi}\check{\sA}^{0,1}(f,f)\leq \frac{1}{2\pi}\liminf_{n\rightarrow \infty} \ell_n \check{\sA}^{\ell_n,1}(f_n,f_n)\leq \liminf_{n\rightarrow \infty} \ell_n \check{\sA}^{\ell_n}(f_n,f_n).  
\]
Hence the first part (a) of Mosco convergence is verified.  Now fix $f\in \check{\mathscr{G}}^0$.  Take a sequence $g_j\in H^{1/2}(\partial T)$ converging to $f$ in $L^2(\partial T)$,  and we have by the definition of $\check{\sA}^{0}$ that
\[
	\lim_{j\rightarrow \infty} \check{\sA}^0(g_j,g_j)=\check{\sA}^0(f,f). 
\]
It follows from \eqref{eq:D2-2} and \eqref{eq:D6} that for any $j\geq 1$,
\[
	\lim_{n\rightarrow \infty}\ell_n\check{\sA}^{\ell_n}(g_j,g_j)=\check{\sA}^0(g_j,g_j).   
\]
Hence we can find an increasing subsequence $\{n_j\}$ such that 
\[
	|\ell_n\check{\sA}^{\ell_n}(g_j,g_j)-\check{\sA}^0(g_j,g_j)|<2^{-j},\quad \forall n\geq n_j. 
\]
Put $f_1=\cdots=f_{n_1-1}=0$ and $f_{n_j}=\cdots =f_{n_{j+1}-1}=g_j$ for $j\geq 1$.  One can easily verify that $f_n$ converges to $f$ in $L^2(\partial T)$ and $\lim_{n\rightarrow \infty}\ell_n\check{\sA}^{\ell_n}(f_n,f_n)=\check{\sA}^0(f,f)$.  Therefore the second part (b) of Mosco convergence is concluded.
\item[(2)]  Let $\{f_n\}$ be a sequence that converges to $f$ weakly in $L^2(\partial T)$.  Then it follows from Proposition~\ref{PROD3}~(1) and Proposition~\ref{PROD4} that 
\[
\begin{aligned}
	\check{\sA}^{\ell_0}(f,f)&=\frac{1}{2\pi}\check{\sA}^{\ell_0,1}(f,f)+\frac{1}{8\pi}\check{\sA}^{\ell_0,2}(f,f)\\
	&\leq \frac{1}{2\pi}\liminf_{n\rightarrow \infty} \check{\sA}^{\ell_n,1}(f_n,f_n)+\frac{1}{8\pi}\liminf_{n\rightarrow \infty}\check{\sA}^{\ell_n,2}(f_n,f_n)\\
	& \leq \liminf_{n\rightarrow \infty} \check{\sA}^{\ell_n}(f_n,f_n).  
\end{aligned}\]
Hence the first part (a) of Mosco convergence is verified.  The second part (b) of Mosco convergence can be derived by the same argument as that in the proof of the first assertion.
\end{itemize}
\end{proof}

We end this section with a remark on the limit of $(\check{\sE}^\ell,\check{\sF}^\ell)$ as $\ell \downarrow 0$ or $\ell \uparrow \infty$.  
Note that the distance between the two components of $\partial \bar{\Omega}_\ell$ is $\pi \ell$.  When $\ell\downarrow 0$,  the two components tend to the common $x_1$-axis.  At a heuristic level,  the first kind of jump enjoyed by the trace process $\check{B}^\ell$,  characterized by the first term in the expression of $\check{\sE}^\ell$, becomes so frequent that its energy diverges as $\ell\downarrow 0$.  Particularly,  $\check{\sE}^\ell$ (or more exactly $\check{\sA}^\ell$) also diverges as $\ell\downarrow 0$.  However when $\ell\uparrow \infty$,  the distance between the two components of $\partial \bar{\Omega}^\ell$ tends to infinity.  Meanwhile the connection of the two components is cut off,  and the limiting process of $\check{B}^\ell$ is a union of two distinct Cauchy processes on the upper and lower real lines respectively. 

\appendix

\section{Mosco convergence}\label{APP}

Let $(\EE^n,\FF^n)$ be a sequence of closed forms on a same Hilbert space $L^2(E,m)$, and $(\EE,\FF)$ be another closed form on $L^2(E,m)$. We always extend the domains of $\EE$ and $\EE_n$ to $L^2(E,m)$ by letting
\[
\begin{aligned}
	\EE(u,u)&:=\infty, \quad u\in L^2(E,m)\setminus \FF, \\ 
	\EE^n(u,u)&:=\infty,\quad u\in L^2(E,m)\setminus \FF^n.
\end{aligned}
\]
In other words, $u\in \FF$ (resp. $u\in \FF^n$) if and only if $\EE(u,u)<\infty$ (resp. $\EE^n(u,u)<\infty$). 
Furthermore, we say $u_n$ converges to $u$ weakly in $L^2(E,m)$, if for any $v\in L^2(E,m)$, $(u_n,v)_m\rightarrow (u,v)_m$ as $n\rightarrow \infty$, and (strongly) in $L^2(E,m)$, if $\|u_n-u\|_{L^2(E,m)}\rightarrow 0$. 

\begin{definition}\label{DEF41}
Let $(\EE^n,\FF^n)$ and $(\EE,\FF)$ be given above. Then $(\EE^n,\FF^n)$ is said to be convergent to $(\EE,\FF)$ in the sense of Mosco, if
\begin{itemize}
\item[(a)] For any sequence $\{u_n:n\geq 1\}\subset L^2(E,m)$ that converges weakly to $u$ in $L^2(E,m)$, it holds that
\begin{equation}\label{eqMoscoa}
	\EE(u,u)\leq \liminf_{n\rightarrow \infty}\EE^n(u_n,u_n). 
\end{equation}
\item[(b)] For any $u\in L^2(E,m)$, there exists a sequence $\{u_n:n\geq 1\}\subset L^2(E,m)$ that converges strongly to $u$ in $L^2(E,m)$ such that
\begin{equation}\label{eqA2}
	\EE(u,u)\geq \limsup_{n\rightarrow \infty}\EE^n(u_n,u_n). 
\end{equation}
\end{itemize}
\end{definition} 

Let $(T^n_t)$ and $(T_t)$ be the semigroups of $(\EE^n, \FF^n)$ and $(\EE,\FF)$ respectively, and $(G^n_\alpha), (G_\alpha)$ be their corresponding resolvents.  The following result is well-known (see \cite{U94}).

\begin{theorem}
Let $(\EE^n, \FF^n), (\EE, \FF)$ be above. Then the following are equivalent:
\begin{itemize}
\item[(1)] $(\EE^n,\FF^n)$ converges to $(\EE,\FF)$ in the sense of Mosco;
\item[(2)] For every $t>0$ and $f\in L^2(E,m)$,  $T^n_tf$ converges to $T_tf$ strongly in $L^2(E,m)$; 
\item[(3)] For every $\alpha>0$ and $f\in L^2(E,m)$,  $G^n_\alpha f$ converges to $G_\alpha f$ strongly in $L^2(E,m)$. 
\end{itemize}
\end{theorem}

\bibliographystyle{abbrv}
\bibliography{TraceBM}

%\section*{Appendix: ABC}

\end{document}